%% file: vertical-dependency.tex
\documentclass[fleqn,11pt]{SelfArx} % Document font size and equations flushed left

\usepackage[english]{babel} % Specify a different language here - english by default
\usepackage{mathpazo}
\usepackage{lipsum} % Required to insert dummy text. To be removed otherwise

\usepackage{amsmath,amssymb, amsthm, mathrsfs}

\usepackage{graphicx}	
\usepackage{subfigure}
\usepackage{float} 
\usepackage{amsfonts, dsfont}
\usepackage{authblk}
\usepackage{algorithm}
\usepackage[noend]{algpseudocode}
\usepackage{tikz,tikz-qtree,tikz-qtree-compat}
\usetikzlibrary{arrows,automata,trees}
\usepackage{etoolbox}
\AtBeginEnvironment{quote}{\small}

\newcommand{\N}{\mathbb{N}}

\newcommand{\frall}{\; \forall \;}
\newcommand{\exs}{\exists \;}
\renewcommand{\epsilon}{\varepsilon}

\newtheorem{theorem}{Theorem}
\newtheorem{lemma}{Lemma}

\newtheorem{example}{Example}
\newtheorem{corollary}{Corollary}
\newtheorem{definition}{Definition}

\definecolor{color1}{HTML}{003366} % Color of the article title and sections
\definecolor{color2}{RGB}{170, 170,170} 
\colorlet{color3}{color1!70!}% Color of the boxes behind the abstract and headings
\definecolor{accent}{RGB}{152,55,64}

\usepackage{hyperref} % Required for hyperlinks
\hypersetup{hidelinks,colorlinks,breaklinks=true,urlcolor=color2,citecolor=color1,linkcolor=color3,bookmarksopen=false,pdftitle={Title},pdfauthor={Author}}

\JournalInfo{Academic Advances of the CTO, Vol. 1, Issue 2, 2017} % Journal information
\Archive{Original Research Article} % Additional notes (e.g. copyright, DOI, review/research article)

\PaperTitle{Vertical Dependency in Sequences of Categorical Random Variables} % Article title

\Authors{Rachel Traylor\footnote{\textit{Office of the CTO, Dell EMC}} , Ph.D. and Jason Hathcock} % Authors
 % Keywords - if you don't want any simply remove all the text between the curly brackets
\newcommand{\keywordname}{Keywords} % Defines the keywords heading name

\Abstract{ This paper develops a more general theory of sequences of dependent categorical random variables, extending the works of Korzeniowski (2013) and Traylor (2017) that studied first-kind dependency in sequences of Bernoulli and categorical random variables, respectively. A more natural form of dependency, sequential dependency, is defined and shown to retain the property of identically distributed but dependent elements in the sequence. The cross-covariance of sequentially dependent categorical random variables is proven to decrease exponentially in the dependency coefficient $\delta$ as the distance between the variables in the sequence increases. We then generalize the notion of \textit{vertical dependency} to describe the relationship between a categorical random variable in a sequence and its predecessors, and define a class of generating functions for such dependency structures. The main result of the paper is that any sequence of dependent categorical random variables generated from a function in the class $\mathscr{C}_{\delta}$ that is \textit{dependency continuous} yields identically distributed but dependent random variables. Finally, a graphical interpretation is given and several examples from the generalized vertical dependency class are illustrated.  }%\lipsum[1]~}

\begin{document}

\flushbottom % Makes all text pages the same height

\maketitle % Print the title and abstract box

\tableofcontents % Print the contents section

\thispagestyle{empty} % Removes page numbering from the first page

%----------------------------------------------------------------------------------------
%	ARTICLE CONTENTS
%----------------------------------------------------------------------------------------

\section*{Introduction} % The \section*{} command stops section numbering

\addcontentsline{toc}{section}{Introduction} % Adds this section to the table of contents

Many statistical tools and distributions rely on the independence of a sequence or set of random variables. The sum of independent Bernoulli random variables yields a binomial random variable. More generally, the sum of independent categorical random variables yields a multinomial random variable. Independent Bernoulli trials also form the basis of the geometric and negative binomial distributions, though the focus is on the number of failures before the first (or $r$th success).~\cite{mathstat} In data science, linear regression relies on independent and identically distributed (i.i.d.) error terms, just to name a few examples.

The necessity of independence filters throughout statistics and data science, although real data rarely is actually independent. Transformations of data to reduce multicollinearity (such as principal component analysis) are commonly used before applying predictive models that require assumptions of independence. This paper aims to continue building a formal foundation of dependency among sequences of random variables in order to extend these into generalized distributions that do not rely on mutual independence in order to better model the complex nature of real data. We build on the works of Korzeniowski~\cite{Korzeniowski2013} and Traylor~\cite{traylor} who both studied first-kind (FK) dependence for Bernoulli and categorical random variables, respectively, in order to define a general class of functions that generate dependent sequences of categorical random variables. 

Section~\ref{sec: background} gives a brief review of the original work by Korzeniowski~\cite{Korzeniowski2013} and Traylor~\cite{traylor}. In section~\ref{sec: Sequential dependency}, a new dependency structure, \textit{sequential dependency} is introduced, and the cross-covariance matrix for two sequentially dependent categorical random variables is derived. Sequentially dependent categorical random variables are identically distributed but dependent.  Section~\ref{subsec: vert dep generators id} generalized the notion of \textit{vertical dependency structures} into a class that encapsulates both the first-kind (FK) dependence of Korzeniowski~\cite{Korzeniowski2013} and Traylor~\cite{traylor} and shows that all such sequences of dependent random variables are identically distributed. We also provide a graphical interpretation and illustrations of several examples of vertical dependency structures.  
%------------------------------------------------

\section{Background}
\label{sec: background}
We repeat a section from~\cite{traylor} in order to give a review of the original first-kind (FK) dependency created by Korzeniowski. 
\begin{figure}[H]
	\centering
	\input{binary-tree.tex}
	\caption{First Kind Dependence for Bernoulli Random Variables}
	\label{fig: bin tree orig}
\end{figure}
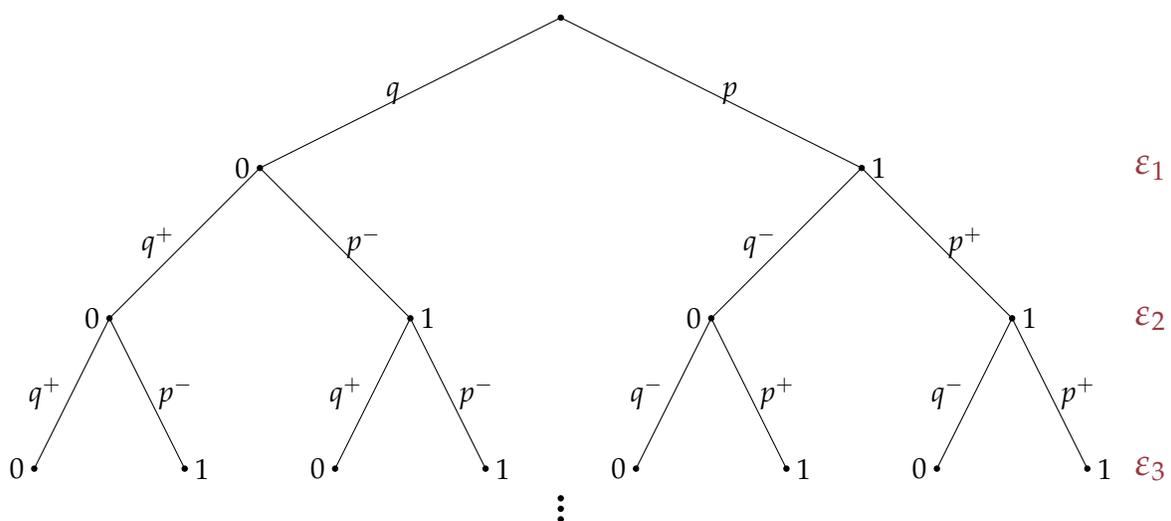

\begin{quote}
Korzeniowski defined the notion of dependence in a way we will refer to here as \textit{dependence of the first kind} (FK dependence). Suppose $(\epsilon_{1},...,\epsilon_{N})$ is a sequence of Bernoulli random variables, and $P(\epsilon_{1} = 1) = p$. Then, for $\epsilon_{i}, i \geq 2$, we weight the probability of each binary outcome toward the outcome of $\epsilon_{1}$, adjusting the probabilities of the remaining outcomes accordingly. 

Formally, let $0 \leq \delta \leq 1$, and $q = 1-p$. Then define the following quantities
\begin{equation}
\begin{aligned}
p^{+} := P(\epsilon_{i} = 1 | \epsilon_{1} = 1) = p + \delta q &\qquad p^{-} := P(\epsilon_{i} = 0 | \epsilon_{1} = 1) = q - \delta q \\
q^{+} := P(\epsilon_{i} = 1 | \epsilon_{1} = 0) = p - \delta p  &\qquad q^{-} := P(\epsilon_{i} = 0 | \epsilon_{1} = 0) = q + \delta p
\end{aligned} 
\end{equation}

Given the outcome $i$ of $\epsilon_{1}$, the probability of outcome $i$ occurring in the subsequent Bernoulli variables $\epsilon_{2}, \epsilon_{3},..., \epsilon_{n}$ is $p^{+}, i = 1$ or $q^{+}, i=0$. The probability of the opposite outcome is then decreased to $q^{-}$ and $p^{-}$, respectively. 

Figure~\ref{fig: bin tree orig} illustrates the possible outcomes of a sequence of such dependent Bernoulli variables. Korzeniowski showed that, despite this conditional dependency, $P(\epsilon_{i} = 1) = p \frall i$. That is, the sequence of Bernoulli variables is identically distributed, with correlation shown to be 
\[Cor(\epsilon_{i}, \epsilon_{j}) = \begin{cases} \delta, & i=1 \\
\delta^{2}, &i \neq j, \quad i,j \geq 2
\end{cases}\]

These identically distributed but correlated Bernoulli random variables yield a Generalized Binomial distribution with a similar form to the standard binomial distribution.
\end{quote}

In~\cite{traylor}, the concept of Bernoulli FK dependence was extended to categorical random variables. That is, given a sequence of categorical random variables with $K$ categories, $P(\epsilon_{1} = i) = p_{i}, i = 1,...,K$. \\
 $P(\epsilon_{j} = i | \epsilon_{1} = i) = p_{i}^{+} = p_{i} + \delta(1-p_{i})$, and 
 $P(\epsilon_{j} = k | \epsilon_{1} = i) = p_{k}^{-} = p_{k} - \delta p_{k}$, $i \neq k$, $k = 1,...,K$.  Traylor proved that FK dependent categorical random variables remained identically distributed, and showed that the cross-covariance matrix of categorical random variables has the same structure as the correlation between FK dependent Bernoulli random variables. In addition, the concept of a generalized binomial distribution was extended to a generalized multinomial distribution. 

In the next section, we will explore a different type of dependency structure, \textit{sequential dependency}.

\section{Sequentially Dependent Categorical Random Variables}
\label{sec: Sequential dependency}	
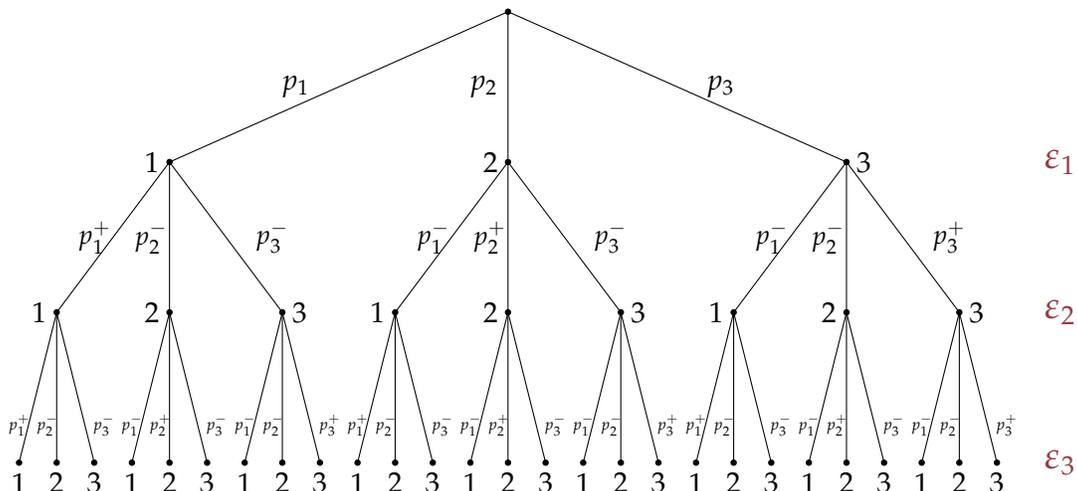
\begin{figure}[H]
	\centering
	\input{seq-dep-tree.tex}
	\caption{Probability Mass Flow of Sequentially Dependent Categorical Random Variables, $K=3$.}
	\label{fig: seq dep tree}
\end{figure}

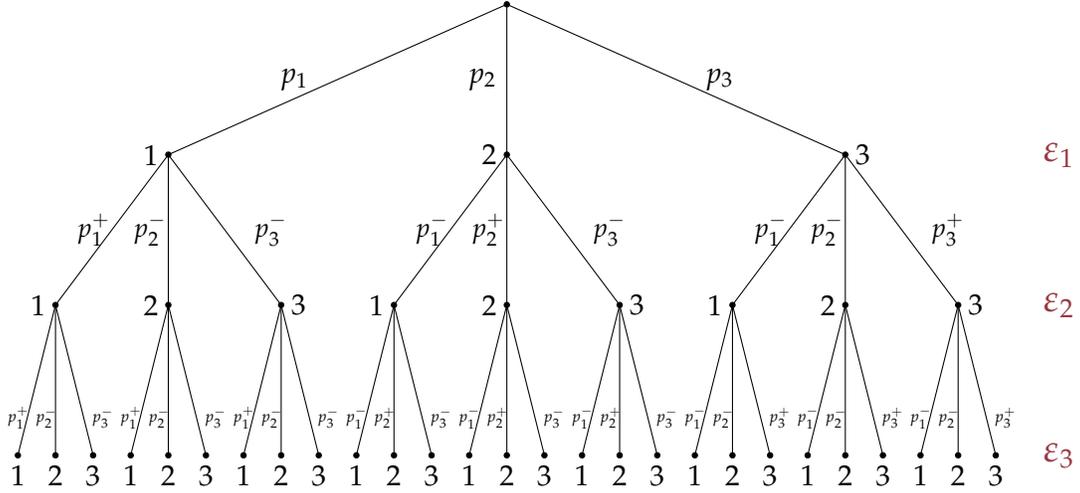
\begin{figure}[H]
	\centering
	\input{first-kind-tree.tex}
	\caption{Probability Mass Flow of FK Dependent Categorical Random Variables, $K=3$.}
	\label{fig: fk dep tree}
\end{figure}

While FK dependence yielded some interesting results, a more realistic type of dependence is \textit{sequential dependence}, where the outcome of a categorical random variable depends with coefficient $\delta$ on the outcome of the variable immediately preceeding it in the sequence. Put formally, if we let $\mathcal{F}_{n} = \{\epsilon_{1},...,\epsilon_{n-1}\}$, then $P(\epsilon_{n}|\epsilon_{1},...,\epsilon_{n-1}) = P(\epsilon_{n}|\epsilon_{n-1}) \neq P(\epsilon_{n})$. That is, $\epsilon_{n}$ only has direct dependence on the previous variable $\epsilon_{n-1}$. We keep the same weighting as for FK-dependence. That is,
\begin{align}
P(\epsilon_{n} = j | \epsilon_{n-1} = j) = p_{j}^{+} = p_{j} + \delta(1-p_{j}), &\qquad P(\epsilon_{n} = j | \epsilon_{n-1} = i) = p_{j}^{-} = p_{j}-\delta p_{j}; j = 1,...,K, i \neq j
\end{align}

As a comparison, for FK dependence, $P(\epsilon_{n}|\epsilon_{1},...,\epsilon_{n-1}) = P(\epsilon_{n}|\epsilon_{1}) \neq P(\epsilon_{n})$. That is, $\epsilon_{n}$ only has direct dependence on $\epsilon_{1}$, and 

\begin{align}
P(\epsilon_{n} = j | \epsilon_{1} = j) = p_{j}^{+} = p_{j} + \delta(1-p_{j}), &\qquad P(\epsilon_{n} = j | \epsilon_{1} = i) = p_{j}^{-} = p_{j}-\delta p_{j}; j = 1,...,K, i \neq j
\end{align}

Let $\epsilon = (\epsilon_{1},...,\epsilon_{n})$ be a sequence of categorical random variables of length $n$ (either independent or dependent) where the number of categories for all $\epsilon_{i}$ is $K$. Denote $\Omega_{n}^{K}$ as the sample space of this random sequence. For example,

	 \[\Omega_{3}^{3} = \{(1,1,1), (1,1,2), (1,1,3), (1,2,1), (1,2,2),...,(3,3,1),(3,3,2),(3,3,3)\}\]

Dependency structures like FK-dependence and sequential dependence change the probability of a sequence $\epsilon$ of length $n$ taking a particular $\omega = (\omega_{1},...,\omega_{n}) \in \Omega_{n}^{K}$. The probability of a particular $\omega \in \Omega_{n}^{K}$ is given by the dependency structure. For example, if the variables are independent, $P((1,2,1)) = p_{1}^{2}p_{2}$. Under FK-dependence, $P((1,2,1)) = p_{1}p_{2}^{-}p_{1}^{+}$, and under sequential dependence, $P((1,2,1)) = p_{1}p_{2}^{-}p_{1}^{-}$. 
See Figures~\ref{fig: seq dep tree} and ~\ref{fig: fk dep tree} for a comparison of the probability mass flows of sequential dependence and FK dependence. 
Sequentially dependent sequences of categorical random variables remain identically distributed but dependent, just like FK-dependent sequences. That is,

\begin{lemma}
	Let $\epsilon = (\epsilon_{1},...,\epsilon_{n})$ be a sequentially dependent categorical sequence of length $n$ with $K$ categories. Then $P(\epsilon_{j} = i) = p_{i}; \quad i = 1,...,K;\quad j = 1,...,n; n \in \N$.
\label{lemma: identically distributed seq dep cat vars}
\end{lemma}

\begin{proof}
	Fix $n \in \N$. Let $\Phi_{n}^{K}(i) = \{\omega \in \Omega_{n}^{K} : \omega_{n} = i\}, i = 1,...,K$. Note that $|\Phi_{n}^{K}(i)| = K^{n-1}$. Then we may partition $\Omega_{n}^{K}$ using these $\Phi_{n}^{K}(i)$: $\Omega_{n}^{K} = \sqcup_{i=1}^{K} \Phi_{n}^{K}(i)$.\footnote{The union is disjoint}
	
	The event that $\epsilon_{n} = i$ is given by $\bigcup_{\omega \in \Phi_{n}^{K}(i)}\omega$, and thus $P(\epsilon_{n} = i) = P\left(\bigcup_{\omega \in \Phi_{n}^{K}(i)}\omega\right), i = 1,...,K$. Each sequence $\omega_{j} \in \Phi_{n}^{K}(i)$ has an associated probability $\pi_{j} = \prod_{l=1}^{n}\pi_{j_{l}}$, where $\pi_{j_{l}}$ is the probability of the $l$th term in the sequence $\omega_{j_l}$. Therefore, 
	\begin{equation}
		P(\epsilon_{n} = i) = \sum_{j=1}^{K^{n-1}}\pi_{j} = \sum_{j=1}^{K^{n-1}}\prod_{l=1}^{n}\pi_{j_{l}}
		\label{eq: seq dep id proof prob}
	\end{equation}
	
	WLOG, assume $i=1$. Then for $n=2$, under sequential dependence, 
		\[P(\epsilon_{2} = 1) = p_{1}p_{1}^{+} + p_{2}p_{1}^{-} + ... + p_{K}p_{1}^{-} = p_{1}\]
	by Lemma 1 of~\cite{traylor}. For $n=3$, 
	\begin{align*}
		P(\epsilon_{3} = 1) &= p_{1}p_{1}^{+}p_{1}^{+} + p_{1}p_{2}^{-}p_{1}^{-} +p_{1}p_{K}^{-}p_{1}^{-} + ... + p_{K}p_{1}^{-}p_{1}^{+} + ... + p_{K}p_{K}^{+}p_{1}^{-} \\
								&= p_{1}^{+}\left(p_{1}p_{1}^{+} + p_{1}^{-}\sum_{j \neq 1}p_{j}\right) + \sum_{i =2}^{K}p_{1}^{-}\left(p_{i}p_{i}^{+} + p_{i}^{-}\sum_{j\neq i}p_{j}\right) \\
								&= p_{1}^{+}p_{1} + p_{1}^{-}\sum_{i=2}^{K}p_{i} \\
								&= p_{1}
	\end{align*}
	
	It is clear that these hold true for $i=2,...,K$. That is, $P(\epsilon_{2} = i) = p_{i} \frall i$ and $P(\epsilon_{3} = i) = p_{i} \frall i$. Now, assume that $P(\epsilon_{n} = i) = p_{i}, i = 1,...,K$. Then, WLOG, we will show that $P(\epsilon_{n+1} = 1) = p_{1}$. 
	We have that $\Phi_{n+1}^{K}(1) = \Omega_{n} \times \{1\}= \left(\sqcup_{i=1}^{K}\Phi_{n}^{K}(i)\right) \times \{1\} = \sqcup_{i=1}^{K}\left(\Phi_{n}^{K}(i) \times \{1\}\right)$
	Therefore,
	\begin{align*}
	P(\epsilon_{n+1} = 1) &= P(\Phi_{n+1}^{K}(1)) \\
									&= \sum_{i=1}^{K}P\left(\Phi_{n}^{K}(i) \times \{1\}\right) \\
									&= p_{1}^{+}P(\Phi_{n}^{K}(i)) + p_{1}^{-}\sum_{i=2}^{K}P(\Phi_{n}^{K}(i)) \\
									&= p_{1}^{+}p_{1} + p_{1}^{-}\sum_{i=2}^{K}p_{i} \\
									&= p_{1}
	\end{align*}
	A similar argument follows for $i=2,....,K$ to conclude that for any $n$, $P(\epsilon_{n} = i) = p_{i}$ under sequential dependence.
\end{proof}

\subsection{Cross-Covariance Matrix}

The $K \times K$ cross-covariance matrix $\Lambda^{m,n}$ of $\epsilon_{m}$ and $\epsilon_{n}$ in a sequentially dependent categorical sequence has entries $\Lambda_{i,j}^{m,n}$ that give the cross-covariance $Cov([\epsilon_{m} = i], [\epsilon_{n} = j])$, where $[\cdot]$ denotes an Iverson bracket. In the FK-dependent case, the entries of $\Lambda^{m,n}$ are given by~\cite{traylor}

	$\Lambda^{1,n}_{ij} = \begin{cases}
\delta p_{i}(1-p_{i}), & i = j \\
-\delta p_{i}p_{j}, & i \neq j
\end{cases}$, $n \geq 2$, and $\Lambda^{m,n}_{ij} = \begin{cases}
\delta^{2}p_{i}(1-p_{i}), & i = j \\
-\delta^{2}p_{i}p_{j}, & i \neq j
\end{cases}$, $n > m,$ $m\neq 1$.

Thus, the cross covariance between any two $\epsilon_{m}$ and $\epsilon_{n}$ in the FK-dependent case is never smaller than $\delta^{2}$ times the independent cross-covariance. In the sequentially dependent case, the cross-covariances of $\epsilon_{m}$ and $\epsilon_{n}$ decrease in powers of $\delta$ as the distance between the two variables in the sequence increases. 
\begin{theorem}[Cross-Covariance of Dependent Categorical Random Variables]
	Denote $\Lambda^{m,n}$ as the $K \times K$ cross-covariance matrix of $\epsilon_{m}$ and $\epsilon_{n}$ in a sequentially dependent sequence of categorical random variables of length $N$, $m \leq n$, and $n \leq N$, defined as $\Lambda^{m,n} = E[(\epsilon_{m} - E[\epsilon_{m}])(\epsilon_{n} - E[\epsilon_{n}])]$. Then the entries of the matrix are given by
	$\Lambda^{m,n}_{ij} = \begin{cases}
	\delta^{n-m}p_{i}(1-p_{i}), & i = j \\
	-\delta^{n-m} p_{i}p_{j}, & i \neq j
	\end{cases}$
\label{thm: CCV categorical}
\end{theorem} 

The pairwise covariance between two Bernoulli variables in a sequentially dependent sequence is given in the following corollary
\begin{corollary}
	Denote $P(\epsilon_{i} = 1) = p; i = 1,...,n$, and let $q=1-p$. Under sequential dependence,\\ $Cov(\epsilon_{m}, \epsilon_{n}) = pq\delta^{n-m}.$
\label{cor: CCV bernoulli}
\end{corollary}

We move the proof of Theorem~\ref{thm: CCV categorical} to Section~\ref{sec: Appendix}.  We give some examples to illustrate. 

\begin{example}[Bernoulli Random Variables]
	If we want to find the covariance between $\epsilon_{2}$ and $\epsilon_{3}$, then we note that the set $S = \{\omega \in \Omega_{3}^{2} : \omega_{2} = 1, \omega_{3} = 1\}$ is given by $S = \{(1,1,1),(0,1,1)\}$. $P(\epsilon_{2} = 1, \epsilon_{3} = 1) = P(S)$. Thus,
	\begin{align*}
		Cov(\epsilon_{2},\epsilon_{3}) &= P(\epsilon_{2} = 1, \epsilon_{3} = 1) - P(\epsilon_{2} = 1)P(\epsilon_{3} = 1) \\
			&= pp^{+}p^{+} + qp^{-}p^{+} - p^{2} \\
			&= p^{+}(pp^{+} + qp^{-}) - p^{2} \\
			&= p(p^{+} - p) \\
			&= pq\delta
	\end{align*}
	
	Similarly, 
  \begin{align*}
  Cov(\epsilon_{1}, \epsilon_{3}) &= P(\epsilon_{1} = 1, \epsilon_{3} = 1) - P(\epsilon_{1} = 1)P(\epsilon_{3} = 1)\\
  												&= (pp^{+}p^{+} + pq^{-}p^{-}) - p^{2} \\
  												&= p((p+\delta q)^{2} + pq(1-\delta)^{2}) - p^{2} \\
  												&= p(p^{2} + pq + \delta^{2}q^{2} + \delta^{2}pq) - p^{2} \\
  												&= p(p(p+q) + \delta^{2}q(q+p)) - p^{2} \\
  												&= p(p + \delta^{2}q) - p^{2} \\
  												&= pq\delta^{2}
  \end{align*}

\end{example}

\begin{example}[Categorical Random Variables]
	Suppose we have a sequence of categorical random variables, where $K =3$. Then 
	$[\epsilon_{m} = i]$ is the Bernoulli random variable with the binary outcome of $1$ if $\epsilon_{m} = i$ and 0 if not. Thus,
	$Cov([\epsilon_{m} = i], [\epsilon_{n} = j]) = P(\epsilon_{m} = i \vee \epsilon_{n} = j) - P(\epsilon_{m} = i)P(\epsilon_{n} = j)$.
	We have shown that every $\epsilon_{n}$ in the sequence is identically distributed, so $P(\epsilon_{m} = i) = p_{i}$ and $P(\epsilon_{n} = j) = 	p_{j}$. 
	
	\begin{align*}
	Cov([\epsilon_{2} = 1], [\epsilon_{3} = 1]) &= (p_{1}p_{1}^{+}p_{1}^{+} + p_{2}p_{1}^{-}p_{1}^{+} + p_{3}p_{1}^{-}p{1}^{+}) - p_{1}^{2} \\
																&= p_{1}^{+}(p_{1}p_{1}^{+} + p_{2}p_{1}^{-} + p_{3}p_{1}^{-}) - p_{1}^{2}  \\
																&= p_{1}^{+}p_{1} - p_{1}^{2} \\
																&= \delta p_{1}(1-p_{1})
	\end{align*}
	
	\begin{align*}
	Cov([\epsilon_{2} = 1], [\epsilon_{3} = 2]) &= (p_{1}p_{1}^{+}p_{2}^{-} + p_{2}p_{1}^{-}p_{2}^{-}+p_{3}p_{1}^{-}p_{2}^{-}) - p_{1}p_{2} \\		
				&= p_{2}^{-}(p_{1}p_{1}^{+} + p_{2}p_{1}^{-} + p_{3}p_{1}^{-})  - p_{1}p_{2}\\
				&= p_{1}p_{2}^{-} - p_{1}p_{2} \\
				&= -\delta p_{1}p_{2}
	\end{align*}
	We may obtain the other entries of the matrix in a similar fashion. So, the cross-covariance matrix for a $\epsilon_{2}$ and $\epsilon_{3}$ with $K=3$ categories is given by 
	\[ \Lambda^{2,3} = \delta \begin{bmatrix}
	   p_{1}(1-p_{1}) & -p_{1}p_{2} & -p_{1}p_{3} \\
	   -p_{1}p_{2} & p_{2}(1-p_{2}) & -p_{2}p_{3} \\
	   -p_{1}p_{3} & -p_{2}p_{3} & p_{3}(1-p_{3})
	\end{bmatrix}\]
	
	Note that if $\epsilon_{2}$ and $\epsilon_{3}$ are independent, then the cross-covariance matrix is all zeros, because $\delta = 0$. 
\end{example}

\section{Dependency Generators}

Both FK-dependency and sequential dependency structures are particular examples of a class of \textit{vertical dependency structures}. We denote the dependency of subsequent categorical random variables on a previous variable in the sequence as a \textit{vertical dependency}. In this section, we define a class of functions that generate vertical dependency structures with the property that all the variables in the sequence are identically distributed but dependent. 

\subsection{Vertical Dependency Generators Produce Identically Distributed Sequences}
\label{subsec: vert dep generators id}
Define the set of functions $\mathcal{C}_{\delta} = \{\alpha: \N_{\geq 2} \to \N : \alpha(n) \geq 1 \wedge \alpha(n) < n \frall n\}$. The mapping defines a direct dependency of $n$ on $\alpha(n)$, denoted $n \stackrel{\delta}{\rightsquigarrow} \alpha(n)$. We have already defined the notion of direct dependency, but we formalize it here. 

\begin{definition}[Direct Dependency]
	Let $\epsilon_{m}$, $\epsilon_{n}$ be two categorical random variables in a sequence, where $m < n$. We say $\epsilon_{n}$ has a \textbf{direct dependency} on $\epsilon_{m}$, denoted $\epsilon_{n} \stackrel{\delta}{\rightsquigarrow} \epsilon_{m}$, if $P(\epsilon_{n}| \epsilon_{n-1},...,\epsilon_{1}) = P(\epsilon_{n} | \epsilon_{m})$. 
\end{definition}

\begin{example}
	For FK-dependence, $\alpha(n) \equiv 1$. That is, for any $n$,  $\epsilon_{n} \stackrel{\delta}{\rightsquigarrow} \epsilon_{1}$
\end{example}

\begin{example}
	The function $\alpha(n) = n-1$ generates the sequential dependency structure of Section~\ref{sec: Sequential dependency}. 
\end{example}

We now define the notion of dependency continuity. 
\begin{definition}[Dependency Continuity]
	A function $\alpha: \N_{\geq 2} \to \N$ is \textit{dependency continuous} if $\frall n \exists j \in \{1,...,n-1\}$ such that $\alpha(n) = j$
\end{definition}

We require that the functions in $\mathcal{C}_{\delta}$ be dependency continuous. Thus, the formal definition of the class of dependency generators $\mathcal{C}_{\delta}$ is 
\begin{definition}[Dependency Generator]
	We define the set of functions $\mathcal{C}_{\delta} = \{\alpha : \N_{\geq 2} \to \N\}$ such that 
	\begin{itemize}
		\item  $\alpha(n) < n$
		\item  $\frall n \exs j \in \{1,...,n-1\} : \alpha(n) = j$ 
	\end{itemize}
and refer to this class as \textbf{dependency generators}.
\end{definition} Each function in $\mathcal{C}_{\delta}$ generates a unique dependency structure for sequences of dependent categorical random variables where the individual elements of the sequence remain identically distributed. 

\begin{theorem}
    Let $\alpha \in \mathcal{C}_{\delta}$. Then for any $n\in \N, n \geq 2$, the dependent categorical random sequence generated by $\alpha$ has identically distributed elements. 
    \label{thm: id dependency generators}
\end{theorem}

\begin{proof}
	Let $\alpha$ be a dependency continuous function in $\mathcal{C}_{\delta}$. Then $\alpha(2) = 1 \Rightarrow \epsilon_{2} \stackrel{\delta}{\rightsquigarrow} \epsilon_{1}$. Thus, $\epsilon_{2}$ and $\epsilon_{1}$ are a sequentially dependent sequence of length $2$ and identically distributed. Now, $\alpha(3) \in \{1,2\}$, and thus either $\epsilon_{3}  \stackrel{\delta}{\rightsquigarrow} \epsilon_{1}$, in which case $\epsilon_{3}$ and $\epsilon_{1}$ are identically distributed, and thus $\epsilon_{1}, \epsilon_{2}$, and $\epsilon_{3}$ are identically distributed, or $\epsilon_{3}  \stackrel{\delta}{\rightsquigarrow} \epsilon_{2}$. But since $\epsilon_{2}  \stackrel{\delta}{\rightsquigarrow} \epsilon_{1}$, it follows that $\epsilon_{3}  \stackrel{\delta}{\rightsquigarrow} \epsilon_{2}\stackrel{\delta}{\rightsquigarrow} \epsilon_{1}$, which is a sequentially dependent sequence of length $3$, and thus all elements are identically distributed. Now, for and $n > 3$, $\alpha(n) = j \in \{1,...,n-1\}$. If $j=1$ or $2$, then we are done. Suppose then that $j > 2$. Then by dependency continuity, $\exists l \in \{1,...,j-1\}$ such that $\alpha(j) = l$. If $l = 1$ or $2$, then $\epsilon_{n}$ is in a sequentially dependent subsequence of length $4$ or $5$, respectively, and is thus identically distributed within the subsequence. If $l \geq 3$, then $\exists m \in \{1,...,l-1\}$ such that $\alpha(l) = m$. We may proceed in this fashion until we reach a $q$ such that $\alpha(q) = 1$ or $2$, which is forced by dependency continuity. Thus, for any $n$, $\epsilon_{n}$ is appended to a sequentially dependent subsequence that is identically distributed among its members and with $\epsilon_{1}$. Thus all subsequences are identically distributed with $\epsilon_{1}$ and hence each other as well. Thus, the entire sequence is identically distributed with $P(\epsilon_{n} = i) = p_{i}, i = 1,...,K \frall n$. 
\end{proof}

\subsection{Graphical Interpretation and Illustrations}

We may visualize the dependency structures generated by $\alpha \in \mathcal{C}_{\delta}$ via directed \textit{dependency graphs}. Each $\epsilon_{n}$ represents a vertex in the graph, and a directed edge connects $n$ to $\alpha(n)$ to represent the direct dependency generated by $\alpha$. This section illustrates some examples and gives a graphical interpretation of the result in Theorem~\ref{thm: id dependency generators}. 

\subsubsection{First-Kind Dependence}

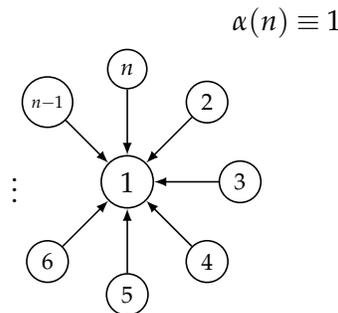
\begin{figure}[H]
	\centering
	\begin{tikzpicture}[->,node distance=1.5cm,semithick]
	\pgfsetarrowsend{latex} 
	\tikzstyle{every state}=[fill=none,text=black] 
	\node  (A) [draw,circle]{$1$}; 
	\node  (B) [draw,circle,above right of=A,scale=.8] {$2$}; 
	\node  (C) [draw,circle,right of=A,scale=.8] {$3$}; 
	\node  (D) [draw,circle,below right of=A,scale=.8] {$4$}; 
	\node  (E) [draw,circle,below of=A,scale=.8] {$5$}; 
	\node  (F) [draw,circle,below left of=A,scale=.8] {$6$}; 
	\node  (G) [left of=A,scale=1] {$\vdots$};
	\node  (H) [draw,circle,above left of=A,scale=.6] {$n\!\!-\!\!1$};  
	\node  (I) [draw,circle,above of=A,scale=.8] {$n$}; 
	\path (B) edge node {} (A) ;
	\path (C) edge node {} (A) ;
	\path (D) edge node {} (A) ;
	\path (E) edge node {} (A) ;
	\path (F) edge node {} (A) ;
	%\path (G) edge node {} (A) ;
	\path (H) edge node {} (A) ;
	\path (I) edge node {} (A) ;
	\node[above right of =1cm] at (B) {$\alpha(n)\equiv 1$}; %optional, can do with caption instead
	\end{tikzpicture} 
	\caption{Dependency Graph for FK Dependence}
	\label{fig: FK dependence graph}
\end{figure}

For FK-dependency, $\alpha(n) \equiv 1$ generates the graph in Figure~\ref{fig: FK dependence graph}. Each $\epsilon_{i}$ depends directly on $\epsilon_{1}$, and thus we see no connections between any other vertices $i, j$ where $j \neq 1$. There are $n-1$ separate subsequences of length $2$ in this graph. 

\subsubsection{Sequential Dependency}

\begin{figure}[H]
	\centering
	\begin{tikzpicture}[->,node distance=1.2cm,semithick] 
	\pgfsetarrowsend{latex} 
	\tikzstyle{every state}=[fill=none,text=black,bend left] 
	\node (A) [draw,circle]{$1$}; 
	\node  (B) [draw,circle,right of=A] {$2$}; 
	\node  (C) [draw,circle,right of=B] {$3$}; 
	\node (D) [right of=C,scale=.7] {$\cdots$};  
	\node  (E) [draw,circle,right of=D,scale=.6] {$n\!-\!1$};
	\node  (F) [draw,circle,right of=E] {$n$};     
	\path (B) edge  node {} (A) ;
	\path (C) edge   node {} (B) ;
	\path (D) edge   node {} (C) ;
	\path (E) edge   node {} (D) ;
	\path (F) edge   node {} (E) ;
	\node[below=0.5cm] at (C) {$\alpha(n)=n-1$}; %optional, can do with caption instead
	\end{tikzpicture} 
	\caption{Dependency Graph for Sequential Dependence}
	\label{fig: seq dep graph}
\end{figure}
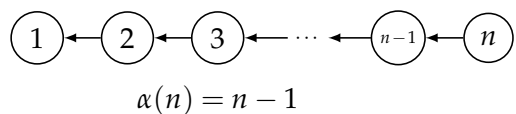

$\alpha(n) = n-1$ generates the sequential dependency structure we studied in Section~\ref{sec: Sequential dependency}. We can see that a path exists from any $n$ back to 1. This  is a visual way to see the result of Theorem~\ref{thm: id dependency generators}, in that if a path exists from any $n$ back to 1, then the variables in that path must be identically distributed. Here, there is only one sequence and no subsequences. 

\subsubsection{A Monotonic Example}

\begin{figure}[H]
	\centering
	\begin{tikzpicture}
	\pgfsetarrowsend{latex}
	\tikzstyle{every state}=[fill=none,text=black] 
	\node (A) [draw,circle]{$1$}; 
	\node (B) [draw,circle,right of=A,node distance=1.5cm,scale=1] {$3$}; 
	\node (C) [draw,circle,left of=A,node distance=1.5cm,scale=1] {$2$}; 
	\node (D) [draw,circle,above left of=C,node distance=1.5cm,scale=.8] {$4$};  
	\node (E) [draw,circle,below of=D,node distance=.8cm,scale=.8] {$5$};
	\node (F) [draw,circle,below of=E,node distance=.8cm,scale=.8] {$6$}; 
	\node (G) [draw, circle,below of=F,node distance=.8cm,scale=.8] {$7$};
	\node (H) [draw,circle,below of=G,node distance=.8cm,scale=.8] {$8$};    
	\node (I) [draw,circle,above right of=B,node distance=2.25cm,scale=.8] {$10$};  
	\node (J) [draw,circle,above of=I,node distance=.8cm,scale=.8] {$9$};
	\node (K) [draw,circle,below of=I,node distance=.8cm,scale=.8] {$11$}; 
	\node (L) [draw,circle,below of=K,node distance=.8cm,scale=.8] {$12$};
	\node (M) [draw,circle,below of=L,node distance=.8cm,scale=.8] {$13$}; 
	\node (N) [draw,circle,below of=M,node distance=.8cm,scale=.8] {$14$};
	\node (O) [draw,circle,below of=N,node distance=.8cm,scale=.8] {$15$};
	\node (P) [draw,circle,above left of=D,node distance=1cm,scale=.6] {$16$};
	\node (Q) [below of=P,node distance=.5cm,scale=.6] {$\vdots$};
	\node (R) [draw,circle,below of=Q,node distance=.5cm,scale=.6] {$24$};
	\path (B) edge   node {} (A) ;
	\path (C) edge   node {} (A) ;
	\path (D) edge   node {} (C) ;
	\path (E) edge   node {} (C) ;
	\path (F) edge   node {} (C) ;
	\path (G) edge   node {} (C) ;
	\path (H) edge   node {} (C) ;
	\path (I) edge   node {} (B) ;
	\path (J) edge   node {} (B) ;
	\path (K) edge   node {} (B) ;
	\path (L) edge node {} (B) ;
	\path (M) edge   node {} (B) ;
	\path (N) edge   node {} (B) ;
	\path (O) edge   node {} (B) ;
	\path (P) edge   node {} (D) ;
	\path (R) edge   node {} (D) ;
	\node (S) [below of=O,node distance=.5cm,scale=.6] {};
	\node[above=1cm] at (A) {$\alpha(n)=\lfloor\sqrt{n}\rfloor$}; %optional, can do with caption instead
	\end{tikzpicture}
	\caption{Dependency Graph for $\alpha(n) = \lfloor \sqrt{n} \rfloor$}
	\label{fig: dependency graph for sqrt(n)}
\end{figure}
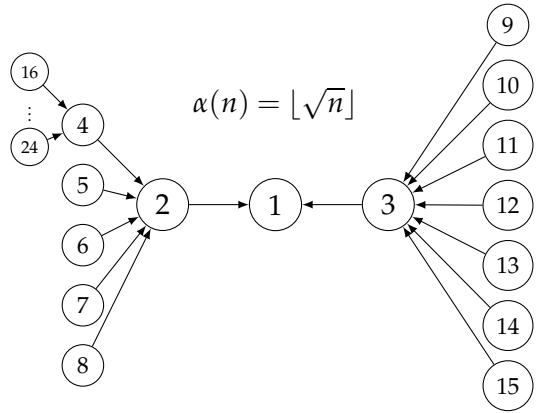

Figure~\ref{fig: dependency graph for sqrt(n)} gives an example of a monotonic function in $\mathcal{C}_{\delta}$ and the dependency structure it generates. Again, note that any $n$ has a path to 1, and the number of subsequences is between 1 and $n-1$. 

\subsubsection{A Nonmonotonic Example }

\begin{figure}[H]
	\centering
	\begin{tikzpicture}
	\pgfsetarrowsend{latex} 
	\tikzstyle{every state}=[fill=none,text=black,scale=.8] 
	\node  (A) [draw,circle]{$1$}; 
	\node  (B) [draw,circle,left of=A,node distance=1cm,scale=.8] {$2$}; 
	\node  (C) [draw,circle,right of=A,node distance=1cm,scale=.8] {$3$}; 
	\node  (D) [draw,circle,above of=A,node distance=1cm,scale=.8] {$5$};  
	\node  (E) [draw,circle,below of=A,node distance=1cm,scale=.8] {$4$};
	\node  (F) [draw,circle,below of=E,node distance=1cm,scale=.6] {$7$}; 
	\node  (G) [draw,circle,below of=F,node distance=1cm,scale=.6] {$13$};
	\node  (H) [draw,circle, below right of=E,node distance=1cm,scale=.6] {$10$};
	\node  (I) [draw,circle,left of=B,node distance=1cm,scale=.6] {$6$};  
	\node  (J) [draw,circle,above left of=D,node distance=1cm,scale=.6] {$8$};
	\node  (K) [draw,circle,above of=D,node distance=1cm,scale=.6] {$9$}; 
	\node  (L) [draw,circle,above right of=D,node distance=1cm,scale=.6] {$12$};
	\node  (M) [draw,circle,right of=C,node distance=1cm,scale=.6] {$11$}; 
	%\node  (N) [below of=M,node distance=1cm,scale=.8] {$14$};
	%\node  (O) [below of=N,node distance=1cm,scale=.8] {$15$};
	%\node  (P) [above left of=D,node distance=2cm] {$16$};
	%\node  (Q) [below of=P,node distance=1cm,scale=.8] {$\vdots$};
	%\node  (R) [below of=Q,node distance=1cm,scale=.8] {$24$};
	\path (B) edge   node {} (A) ;
	\path (C) edge   node {} (A) ;
	\path (D) edge   node {} (A) ;
	\path (E) edge   node {} (A) ;
	\path (F) edge   node {} (E) ;
	\path (G) edge   node {} (F) ;
	\path (H) edge   node {} (E) ;
	\path (I) edge   node {} (B) ;
	\path (J) edge   node {} (D) ;
	\path (K) edge   node {} (D) ;
	\path (L) edge   node {} (D) ;
	\path (M) edge   node {} (C) ;
	%\path (N) edge   node {} (B) ;
	%\path (O) edge   node {} (B) ;
	%\path (P) edge   node {} (D) ;
	%\path (R) edge   node {} (D) ;
%	\node[above right=1cm] at (A) {$\alpha(n)=\left\lfloor\frac{\sqrt{n}}{2}\left(\sin(n)\right) + \frac{n}{2}\right\rfloor$}; %optional, can do with caption instead
	\end{tikzpicture} 
	\caption{Dependency Graph for $\alpha(n) = \left\lfloor\frac{\sqrt{n}}{2}\left(\sin(n)\right) + \frac{n}{2}\right\rfloor$}
	\label{fig: graph for nonmonotonic}
\end{figure}
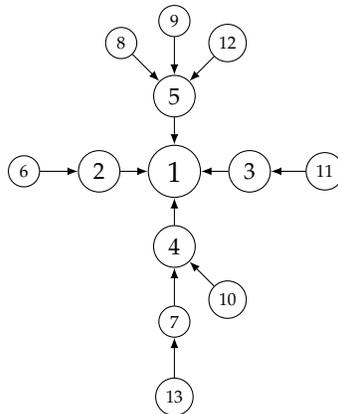

Figure~\ref{fig: graph for nonmonotonic} illustrates a more contrived example where the function is nonmonotonic. It is neither increasing nor decreasing. The previous examples have all been nondecreasing. 

\subsubsection{A Prime Example}

Let $\alpha \in \mathcal{C}_{\delta}$ be defined in the following way. Let $p_{m}$ be the $m$th prime number, and let $\{kp_{m}\}$ be the set of all postive integer multiples of $p_{m}$. Then the set $\mathscr{P}_{m} = \{kp_{m}\} \setminus \cup_{i=1}^{m-1}\{kp_{i}$ gives a disjoint partitioning of $\N$.  That is, $\N = \sqcup_{m}\mathscr{P}_{m}$, and thus for any $n \in \N$, $n \in \mathscr{P}_{m}$ for exactly one $m$. Now let $\alpha(n) = m$. Thus, the function is well-defined. We may write $\alpha: \N_{\geq 2} \to \N$ as $\alpha[\mathscr{P}_{m}] = m$. As an illustration,
\begin{align*}
\alpha[\{2k\}] &= 1 \\
\alpha[\{3k\}\setminus \{2k\}] &= 2 \\
\alpha[\{5k\}\setminus (\{2k\}\cup \{3k\})] &= 3 \\
&\hspace{.17cm}\vdots\\
\alpha[\{kp_{m}\}\setminus (\cup_{i=1}^{m-1}\{kp_{i}\})] &= m
\end{align*}

\section{Conclusion}
\label{sec: conclusion}

This paper has extended the concept of dependent random sequences first put forth in the works of Korzeniowski~\cite{Korzeniowski2013} and Traylor~\cite{traylor} and developed a generalized class of vertical dependency structures. The sequential dependency structure was studied extensively, and a formula for the cross-covariance obtained. The class of dependency generators was defined in Section~\ref{subsec: vert dep generators id} and shown to always produce a unique dependency structure for any $\alpha \in \mathscr{C}_{\delta}$ in which the a sequence of categorical random variables under that $\alpha$ is identically distributed but dependent. We provided a graphical interpretation of this class, and illustrated with several key examples. 
\section{Appendix}
\label{sec: Appendix}

\subsection{Proof of Theorem~\ref{thm: CCV categorical}}
The proof can be divided into two statements: 
\[\begin{cases}
\text{Statement 1}:& Cov([\epsilon_{m} = i], [\epsilon_{n} = i]) = p_{i}(1-p_{i})\delta^{n-m} \\
\text{Statement 2}:& Cov([\epsilon_{m} = i], [\epsilon_{n} = j]) = -p_{i}p_{j}\delta^{n-m}; \quad m > 1, n > m 
\end{cases}\]

We prove Statement 1 first for $m=1$. It suffices to show that $P(\epsilon_{1} = i, \epsilon_{n} = i) = p_{i}(p_{i} + \delta^{n-1}(1-p_{i}))$. This will be shown via induction, though the base case is 6. $n=2,..,5$ are calculated directly.  Fix the number of categories $K$, and denote the sample space of a sequentially dependent categorical sequence of length $n$ with $K$ categories by $\Omega_{n}$. Fix $i \in \{1,...,K\}$
Denote the following sets: 
\begin{align*}
\Theta_{n}^{(i)} &= \{\omega \in \Omega_{n} : \omega_{1} = \omega_{n} = i\}\\
\Phi_{n}^{(i)} &= \{\omega \in \Omega_{n} : \omega_{n} = i\}
\end{align*}

For $n=2$, $P(\epsilon_{1} = i, \epsilon_{2} = i) = p_{i}p_{i}^{+} = p_{i}(p_{i} + \delta(1-p_{i}))$, $i=1,...,K$. For $n=3$, $\Theta_{3}^{i} = \{i\} \times \Phi_{2}^{(i)}$. Then 
\[P(\epsilon_{1} = i, \epsilon_{3}= i) = P(\Theta_{3}^{i}) = p_{i}(p_{i}^{+}p_{i}^{+}) + p_{i}\left(\sum_{j \neq i}p_{j}^{-}\right)p_{i}^{-} = p_{i}(p_{i} + (1-p_{i})\delta^{2})\]
For $n=4$,
\begin{align*}
\Theta_{4}^{(i)} &= \{i\} \times \Phi_{3}^{(i)} \\
			&= \{i\} \times \left(\bigcup_{j=1}^{K}\{j\}\right) \times \Phi_{2}^{(i)} \\
			&= \left[\{i\} \times \Theta_{3}^{(i)}\right] \cup \left[\bigcup_{j \neq i}\{i\} \times \{j\} \times \Phi_{2}^{(i)}\right]
\end{align*}

Let $\epsilon_{(-k)}$ be a sequence with the first $k$ terms missing. 

\begin{align*}
P(\Theta_{4}^{(i)}) &= P(\epsilon_{(-1)} \in \Theta_{3}^{(i)}|\epsilon_{1} = i)P(\epsilon_{1} = i) + \sum_{j \neq i}P(\epsilon_{(-2)} \in \Phi_{2}^{(i)}|\epsilon_{1} = i, \epsilon_{2} = j)P(\epsilon_{2} = j | \epsilon_{1} = i)P(\epsilon_{1} = i) \\
 &= p_{i}p_{i}^{+}(p_{i} + \delta^{2}(1-p_{i}) + \sum_{j \neq i})p_{i}p_{j}^{-}p_{i}^{-}(1+\delta) \\
 &= p_{i}p_{i}^{+}(p_{i} + \delta^{2}(1-p_{i}) + p_{i}(1-p_{i})(1-\delta-\delta^{2} + \delta^{3}) \\
 &= p_{i}(p_{i} + (1-p_{i})\delta^{3})
\end{align*}

The calculations for $n=5$ proceed in a similar fashion. $P(\Theta_{5}^{(i)}) = p_{i}(p_{i} + (1-p_{i})\delta^{4})$.
Then for $n=6$, which is our true base case for the inductive hypothesis,
\begin{align*}
\Theta_{6}^{(i)} &= \{i\} \times \Theta_{5} \cup \left(\bigcup_{j\neq i}\{i\} \times \{j\} \times \Phi_{4}^{(i)}\right) \\
\end{align*}

We already know that $P(\epsilon_{(-1)} \in \Theta_{5}|\epsilon_{1}= i) = p_{i}p_{i}^{+}(p_{i} + (1-p_{i})\delta^{4})$. Now, fix $j$. Then we may partition $\{i\} \times \{j\} \times \Phi_{4}^{(i)}$. 
\begin{equation}
\begin{aligned}
\{i\} \times \{j\} \times \Phi_{4}^{(i)} &= \left[\{i\} \times \{j\} \times \Theta_{4}^{(i)}\right] \cup \left[\{i\} \times \{j\} \times \{j\} \times \Phi_{3}^{i}\right] \cup \left(\bigcup_{k \neq i,j} \left[\{i\} \times \{j\} \times \{k\} \times \Phi_{3}^{(i)}\right]\right)
\end{aligned}
\end{equation}

Then we calculate the probability of each partition separately.
\begin{equation}
\begin{aligned}
P\left(\{i\} \times \{j\} \times \Theta_{4}^{(i)}\right) &= p_{i}p_{j}^{-}p_{i}^{-}(p_{i} + (1-p_{i})\delta^{3}) 
\end{aligned}
\label{eq: i j Theta 4}
\end{equation}
\begin{equation}
\begin{aligned}
P\left(\{i\} \times \{j\} \times \{j\} \times \Phi_{3}^{i}\right) &= p_{i}p_{j}^{-}p_{j}^{+}P(\epsilon_{(-3)} \in \Phi_{3}^{(i)} | \epsilon_{1} = i, \epsilon_{2} = \epsilon_{3} = j)  \\
&= p_{i}(p_{i}p_{j}(1-\delta)(1-\delta^{3})(p_{j} + (1-p_{j})\delta)) 
\end{aligned}
\label{eq: i j j Phi 3}
\end{equation}
and for fixed $k \neq i,j$,
\begin{equation}
\begin{aligned}
P\left(\{i\} \times \{j\} \times \{k\} \times \Phi_{3}^{(i)}\right) &= p_{i}p_{j}^{-}p_{k}^{-}P(\epsilon_{(-3)} \in \Phi_{3})| \epsilon_{1} = i, \epsilon_{2} = j, \epsilon_{3}= k) \\
&= p_{i}\left(\prod_{l=1}^{K}p_{l}\right)(1-\delta)^{2}(1-\delta^{3})
\end{aligned}
\label{eq: i j k Phi3}
\end{equation}

Then $P\left(\bigcup_{k \neq i,j} \left[\{i\} \times \{j\} \times \{k\} \times \Phi_{3}^{(i)}\right]\right)$ is obtained by summing the probabilities in~\eqref{eq: i j Theta 4} -~\eqref{eq: i j k Phi3} over all $j \neq i$. That is,

\begin{align*}
P\left(\bigcup_{k \neq i,j} \left[\{i\} \times \{j\} \times \{k\} \times \Phi_{3}^{(i)}\right]\right) &= \sum_{\substack{j=1\\j\neq i}}^{K}\left(p_{i}p_{j}^{-}p_{i}^{-}(p_{i} + (1-p_{i})\delta^{3}) +  p_{i}(p_{i}p_{j}(1-\delta)(1-\delta^{3})(p_{j} + (1-p_{j})\delta))  \right. \\
&\qquad\left.+ \sum_{l \neq i,j}p_{i}\left(\prod_{l=1}^{K}p_{l}\right)(1-\delta)^{2}(1-\delta^{3})\right) \\
&= p_{i}(1-p_{i})(1-\delta -\delta^{4} + \delta^{5})
\end{align*}

Therefore, 
\begin{align*}
P\left(\Theta_{6}^{(i)}\right) &= p_{i}p_{i}^{+}(p_{i}+(1-p_{i})\delta^{4}) +  p_{i}(1-p_{i})(1-\delta -\delta^{4} + \delta^{5}) \\
			&= p_{i}(p_{i} + (1-p_{i})\delta^{5})
\end{align*}

For the inductive hypothesis, assume for $6 \leq m \leq n$, 

\begin{equation}
P(\epsilon_{(-1)} \in \Theta_{n-1}|\epsilon_{1} = i) = p_{i}^{+}(p_{i} + (1-p_{i})\delta^{n-2})
\label{eq: ind hyp 1}
\end{equation}

\begin{equation}
P(\epsilon_{(-1)} \in \Theta_{n-1}|\epsilon_{1} \neq i) = p_{i}^{-}(p_{i} + (1-p_{i})\delta^{n-2})
\label{eq: ind hyp 2}
\end{equation}

\begin{equation}
P(\epsilon_{(-3)} \in \Phi_{n-3}^{(i)} | \epsilon_{1} = i, \epsilon_{2} = \epsilon_{3} = j) = p_{j}(1-\delta^{n-3})
\label{eq: ind hyp 3}
\end{equation}

and for $k \neq i,j$,
\begin{equation}
P(\epsilon_{(-3)} \in \Phi_{n-3}^{(i)} | \epsilon_{1} = i, \epsilon_{2} = j\epsilon_{3} = k) = (1-\delta^{n-3})\prod_{l \neq j,k}p_{l}
\label{eq: ind hyp 4}
\end{equation}

Then $P(\epsilon_{1}= i, \epsilon_{n+1} = i) = P(\Theta_{n+1}^{(i)})$. Now, 
\begin{align*}
\Theta_{n+1}^{(i)} &= \{i\} \times \Phi_{n}^{(i)}\\
					&= \left(\{i\} \times \Theta_{n}\right) \cup \left(\bigcup_{j \neq i}\{i\} \times \{j\}\times \Phi_{n-1}\right)		
\end{align*}
Then we can break $\{i\} \times \Theta_{n}$ down further and express it as 
\begin{align}
\{i\} \times \Theta_{n} &= \left(\{i\} \times \{i\} \times \Theta_{n-1}\right) \cup \left(\bigcup_{k\neq i}\{i\} \times \{k\} \times \Phi_{n-1}\right)
\end{align}

We can then see that $P(\{i\} \times \{i\} \times \Theta_{n-1}) = p_{i}p_{i}^{+}(p_{i} + (1-p_{i})\delta^{n-2})$ by the inductive hypothesis. Then, we break down (for fixed $k$) $\{i\} \times \{k\} \times \Phi_{n-1}$ into 
\begin{align}
\{i\} \times \{k\} \times \Phi_{n-1} &=\left( \{i\} \times \{k\} \times \Theta_{n-1} \right) \cup \left(\bigcup_{l \neq i} \{i\} \times \{k\} \times \{l\} \times \Phi_{n-2}\right)
\end{align}

By the inductive hypothesis, we may calculate (in a similar manner as for the base case $n=6$) that 
\[P(\bigcup_{k\neq i}\{i\} \times \{k\} \times \Phi_{n-1}) = p_{i}(1-p_{i})(1-\delta - \delta^{n-2} + \delta^{n-1})\]
and therefore 
\[P(\{i\} \times \Theta_{n}) = p_{i}p_{i}^{+}(p_{i} + (1-p_{i})\delta^{n-2}) + p_{i}(1-p_{i})(1-\delta - \delta^{n-2} + \delta^{n-1}) = p_{i}(p_{i} + (1-p_{i})\delta^{n-1})\]

In a similar manner, we can break down $\{i\} \times \{j\}\times \Phi_{n-1}$ for fixed $j$, and, using the inductive hypotheses (and some very tedious arithmetic), calculate that 
\[P\left(\bigcup_{k\neq i}\{i\} \times \{k\} \times \Phi_{n-1}\right) = p_{i}(1-p_{i})(1-\delta -\delta^{n-1} + \delta^{n})\]

Combining yields $P(\Theta_{n+1}^{(i)} ) = p_{i}(p_{i} + (1-p_{i})\delta^{n})$ and the statement is proven. 

To prove the Statement 1 for $m > 1$, we abuse notation and redefine the set $\Theta_{n-m}^{(i)} = \{\omega \in \Omega_{n} : \omega_{m} = i, \omega_{n} = i\}$. Then it also suffices to show that $P(\Theta_{n-m}^{(i)}) = p_{i}(1-p_{i})\delta^{n-m}$. 
\[\Theta_{n-m}^{(i)} = \Omega_{m} \times \Theta_{n-m}\]
where $\Omega_{m}$ consists of sequences (or subsequences) of length $m$ and ending at index $m-1$. Then 
\begin{align}
\Omega_{m} \times \Theta_{n-m}^{(i)} &= \bigcup_{j=1}^{K}\Phi_{m}^{(j)} \times \Theta_{n-m}^{(i)} 
\end{align}
Now, since the sequential variables are identically distributed, $P(\Phi_{m}^{j}) = P(\epsilon_{m-1} = j) = p_{j}$. Thus,

\begin{align*}
P(\Theta_{n-m}^{(i)}) &= p_{i}(p_{i}^{+}(p_{i} + (1-p_{i})\delta^{n-m})) + \sum_{j \neq i}p_{j}p_{i}^{-}(p_{i} + (1-p_{i})\delta^{n-m})) \\
		&= p_{i}(p_{i} + (1-p_{i})\delta^{n-m})
\end{align*}

The proof of Statement 2 follows in a similar fashion.

\phantomsection
\section*{Acknowledgments} % The \section*{} command stops section numbering

\addcontentsline{toc}{section}{Acknowledgments} % Adds this section to the table of contents

The authors are grateful to the DPD CTO of Dell EMC and Stephen Manley for continued funding and support. 

%----------------------------------------------------------------------------------------
%	REFERENCE LIST
%----------------------------------------------------------------------------------------
\phantomsection
 
 \bibliographystyle{plain}
 \bibliography{vertical-dependency}

%----------------------------------------------------------------------------------------

\end{document}

%% file: binary-tree.tex
\begin{tikzpicture}
	\node [below] at (7,0) {\huge$\vdots$};
	%%%%%		TOP			%%%%%
	\draw[fill] (7,6) circle [radius=1pt];
	%%%%%		LEVEL 1		%%%%%	
	\foreach \x in {3,11}{%
		\draw [fill] (\x,4) circle [radius=1pt];
		\draw (7,6) -- (\x,4);
		}
	%%%%%		LEVEL 2		%%%%%
	\foreach \x in {1,5,9,13}{%
		\draw [fill] (\x,2) circle [radius=1pt];}
	\foreach \x in {1,5}{%
		\draw (3,4) -- (\x,2);}
	\foreach \x in {9,13}{%
		\draw (11,4) -- (\x,2);}
	%%%%%		LEVEL 3		%%%%%
	\foreach \x in {0,2,4,6,8,10,12,14}{
		\draw [fill] (\x,0) circle [radius=1pt];}
	\foreach \x in {0,2} {
		\draw (1,2) -- (\x,0);}
	\foreach \x in {4,6} {
		\draw (5,2) -- (\x,0);}
	\foreach \x in {8,10} {
		\draw (9,2) -- (\x,0);}
	\foreach \x in {12,14} {
		\draw (13,2) -- (\x,0);}
	%%%%%	NUMERALS		%%%%%
	\foreach \x in {0,4,8,12}{%
		\node [left] at (\x,0) {$0$};
		\node [right] at(\x+2,0) {$1$};
		}
	\foreach \x in {1,9}{%
		\node [left] at (\x,2) {$0$};
		\node [right] at (\x+4,2) {$1$};}
	
	\node [left] at (3,4) {$0$};
	\node [right] at (11,4) {$1$};
	%%%%%	Ps & Qs		%%%%%
	\node [left] at (5,5) {$q$};
	\node [right] at (9,5) {$p$};

	\node [left] at (2,3) {$q^+$};
	\node [right] at (4,3) {$p^-$};
	\node [left] at (10,3) {$q^-$};
	\node [right] at (12,3) {$p^+$};

	\node [left] at (.5,1) {$q^+$};
	\node [right] at (1.5,1) {$p^-$};
	\node [left] at (4.5,1) {$q^+$};
	\node [right] at (5.5,1) {$p^-$};
	\node [left] at (8.5,1) {$q^-$};
	\node [right] at (9.5,1) {$p^+$};
	\node [left] at (12.5,1) {$q^-$};
	\node [right] at (13.5,1) {$p^+$};
	%%%%%	EPSILONS		%%%%%
	\node (e1) [right] at (14.5,4) {\Large\textcolor{accent}{$\varepsilon_1$}};
	\node (e2) [right] at (14.5,2) {\Large\textcolor{accent}{$\varepsilon_2$}};
	\node (e3) [right] at (14.5,0) {\Large\textcolor{accent}{$\varepsilon_3$}};
\end{tikzpicture}

%% file: seq-dep-tree.tex
\begin{tikzpicture}
%%%%%		SEQUENTIAL DEPENDENCE	%%%%%
%%%%%	UNCOMMENT LINE JUST BELOW TO SEE GRIDLINES	%%%%%
%\draw (0,0) grid (15,6);
%%%%%		TOP			%%%%%
\draw[fill] (7.5,6) circle [radius=1pt];
%%%%%		LEVEL 1		%%%%%	
\draw[fill] (3,4) circle [radius=1pt];
\draw[fill] (7.5,4) circle [radius=1pt];
\draw[fill] (12,4) circle [radius=1pt];
\foreach \x in {3,7.5,12}{%
	\draw (7.5,6) -- (\x,4);
}
%%%%%		LEVEL 2		%%%%%
\draw[fill] (1.5,2) circle [radius=1pt];
\draw[fill] (3,2) circle [radius=1pt];
\draw[fill] (4.5,2) circle [radius=1pt];
\foreach \x in {1.5,3,4.5}{%
	\draw (3,4) -- (\x,2);
}
\draw[fill] (6,2) circle [radius=1pt];
\draw[fill] (7.5,2) circle [radius=1pt];
\draw[fill] (9,2) circle [radius=1pt];
\foreach \x in {6,7.5,9}{%
	\draw (7.5,4) -- (\x,2);
}
\draw[fill] (10.5,2) circle [radius=1pt];
\draw[fill] (12,2) circle [radius=1pt];
\draw[fill] (13.5,2) circle [radius=1pt];
\foreach \x in {10.5,12,13.5}{%
	\draw (12,4) -- (\x,2);
}
%%%%%		LEVEL 3		%%%%%
\draw[fill] (1,0) circle [radius=1pt];
\draw[fill] (1.5,0) circle [radius=1pt];
\draw[fill] (2,0) circle [radius=1pt];
\foreach \x in {1,1.5,2}{%
	\draw (1.5,2) -- (\x,0);
}
\draw[fill] (2.5,0) circle [radius=1pt];
\draw[fill] (3,0) circle [radius=1pt];
\draw[fill] (3.5,0) circle [radius=1pt];
\foreach \x in {2.5,3,3.5}{%
	\draw (3,2) -- (\x,0);
}
\draw[fill] (4,0) circle [radius=1pt];
\draw[fill] (4.5,0) circle [radius=1pt];
\draw[fill] (5,0) circle [radius=1pt];
\foreach \x in {4,4.5,5}{%
	\draw (4.5,2) -- (\x,0);
}
\draw[fill] (5.5,0) circle [radius=1pt];
\draw[fill] (6,0) circle [radius=1pt];
\draw[fill] (6.5,0) circle [radius=1pt];
\foreach \x in {5.5,6,6.5}{%
	\draw (6,2) -- (\x,0);
}
\draw[fill] (7,0) circle [radius=1pt];
\draw[fill] (7.5,0) circle [radius=1pt];
\draw[fill] (8,0) circle [radius=1pt];
\foreach \x in {7,7.5,8}{%
	\draw (7.5,2) -- (\x,0);
}
\draw[fill] (8.5,0) circle [radius=1pt];
\draw[fill] (9,0) circle [radius=1pt];
\draw[fill] (9.5,0) circle [radius=1pt];
\foreach \x in {8.5,9,9.5}{%
	\draw (9,2) -- (\x,0);
}
\draw[fill] (10,0) circle [radius=1pt];
\draw[fill] (10.5,0) circle [radius=1pt];
\draw[fill] (11,0) circle [radius=1pt];
\foreach \x in {10,10.5,11}{%
	\draw (10.5,2) -- (\x,0);
}
\draw[fill] (11.5,0) circle [radius=1pt];
\draw[fill] (12,0) circle [radius=1pt];
\draw[fill] (12.5,0) circle [radius=1pt];
\foreach \x in {11.5,12,12.5}{%
	\draw (12,2) -- (\x,0);
}
\draw[fill] (13,0) circle [radius=1pt];
\draw[fill] (13.5,0) circle [radius=1pt];
\draw[fill] (14,0) circle [radius=1pt];
\foreach \x in {13,13.5,14}{%
	\draw (13.5,2) -- (\x,0);
}
%%%%%	EPSILONS		%%%%%
\node (e1) [right] at (14.5,4) {\Large\textcolor{accent}{$\varepsilon_1$}};
\node (e2) [right] at (14.5,2) {\Large\textcolor{accent}{$\varepsilon_2$}};
\node (e3) [right] at (14.5,0) {\Large\textcolor{accent}{$\varepsilon_3$}};
%%%%%	Ps			%%%%%
%%%%% 	TOP			%%%%%
\node [left] at (5,5) {$p_1$};
\node [left] at (7.5,5) {$p_2$};
\node [right] at (10,5) {$p_3$};

%%%%%	SECOND SET	%%%%%
\node at (2,3) {\small$p_1^+$};
\node at (2.75,3) {\small$p_2^-$};
\node [right] at (4,3) {\small$p_3^-$};

\node at (6.5,3) {\small$p_1^-$};
\node at (7.25,3) {\small$p_2^+$};
\node [right] at (8.5,3) {\small$p_3^-$};

\node at (11,3) {\small$p_1^-$};
\node at (11.75,3) {\small$p_2^-$};
\node [right] at (13,3) {\small$p_3^+$};

%%%%%	THIRD SET	%%%%%
\foreach \x in {1,5.5,10}{
	\node at (\x,.5) {\tiny$p_1^+$};
	\node at (\x+.5-.125,.5) {\tiny$p_2^-$};
	\node at (\x+1.125,.5) {\tiny$p_3^-$};
	}
\foreach \x in {2.5,7,11.5}{
	\node at (\x,.5) {\tiny$p_1^-$};
	\node at (\x+.5-.125,.5) {\tiny$p_2^+$};
	\node at (\x+1.125,.5) {\tiny$p_3^-$};
	}
\foreach \x in {4,8.5,13}{
		\node at (\x,.5) {\tiny$p_1^-$};
		\node at (\x+.5-.125,.5) {\tiny$p_2^-$};
		\node at (\x+1.125,.5) {\tiny$p_3^+$};
	}
	%%%%%	NUMERALS		%%%%%
\foreach \x in {1,2.5,4,5.5,7,8.5,10,11.5,13}{%
	\node [below] at (\x,0) {$1$};
}
\foreach \x in {1.5,3,4.5,6,7.5,9,10.5,12,13.5}{%
	\node [below] at (\x,0) {$2$};
}
\foreach \x in {2,3.5,5,6.5,8,9.5,11,12.5,14}{%
	\node [below] at (\x,0) {$3$};
}
\node [left] at (3,4) {$1$};
\node [left] at (7.5,4) {$2$};
\node [right] at (12,4) {$3$};
\foreach \x in {1.5,6,10.5}{%
	\node [left] at (\x,2) {$1$};
}
\foreach \x in {3,7.5,12}{%
	\node [left] at (\x,2) {$2$};
}
\foreach \x in {4.5,9,13.5}{%
	\node [right] at (\x,2) {$3$};
}
\end{tikzpicture}

%% file: first-kind-tree.tex
\begin{tikzpicture}
	%%%%%		FIRST-KIND DEPENDENCE	%%%%%
	%%%%%	UNCOMMENT LINE JUST BELOW TO SEE GRIDLINES	%%%%%
	%\draw (0,0) grid (15,6);
	%%%%%	DOTS & LINES	%%%%%
	%%%%%		TOP			%%%%%
	\draw[fill] (7.5,6) circle [radius=1pt];
	%%%%%		LEVEL 1		%%%%%	
	\draw[fill] (3,4) circle [radius=1pt];
	\draw[fill] (7.5,4) circle [radius=1pt];
	\draw[fill] (12,4) circle [radius=1pt];
	\foreach \x in {3,7.5,12}{%
		\draw (7.5,6) -- (\x,4);
		}
	%%%%%		LEVEL 2		%%%%%
	\draw[fill] (1.5,2) circle [radius=1pt];
	\draw[fill] (3,2) circle [radius=1pt];
	\draw[fill] (4.5,2) circle [radius=1pt];
	\foreach \x in {1.5,3,4.5}{%
		\draw (3,4) -- (\x,2);
		}
	\draw[fill] (6,2) circle [radius=1pt];
	\draw[fill] (7.5,2) circle [radius=1pt];
	\draw[fill] (9,2) circle [radius=1pt];
	\foreach \x in {6,7.5,9}{%
		\draw (7.5,4) -- (\x,2);
		}
	\draw[fill] (10.5,2) circle [radius=1pt];
	\draw[fill] (12,2) circle [radius=1pt];
	\draw[fill] (13.5,2) circle [radius=1pt];
	\foreach \x in {10.5,12,13.5}{%
		\draw (12,4) -- (\x,2);
		}
	%%%%%		LEVEL 3		%%%%%
	\draw[fill] (1,0) circle [radius=1pt];
	\draw[fill] (1.5,0) circle [radius=1pt];
	\draw[fill] (2,0) circle [radius=1pt];
	\foreach \x in {1,1.5,2}{%
		\draw (1.5,2) -- (\x,0);
		}
	\draw[fill] (2.5,0) circle [radius=1pt];
	\draw[fill] (3,0) circle [radius=1pt];
	\draw[fill] (3.5,0) circle [radius=1pt];
	\foreach \x in {2.5,3,3.5}{%
		\draw (3,2) -- (\x,0);
		}
	\draw[fill] (4,0) circle [radius=1pt];
	\draw[fill] (4.5,0) circle [radius=1pt];
	\draw[fill] (5,0) circle [radius=1pt];
	\foreach \x in {4,4.5,5}{%
		\draw (4.5,2) -- (\x,0);
		}
	\draw[fill] (5.5,0) circle [radius=1pt];
	\draw[fill] (6,0) circle [radius=1pt];
	\draw[fill] (6.5,0) circle [radius=1pt];
	\foreach \x in {5.5,6,6.5}{%
		\draw (6,2) -- (\x,0);
		}
	\draw[fill] (7,0) circle [radius=1pt];
	\draw[fill] (7.5,0) circle [radius=1pt];
	\draw[fill] (8,0) circle [radius=1pt];
	\foreach \x in {7,7.5,8}{%
		\draw (7.5,2) -- (\x,0);
		}
	\draw[fill] (8.5,0) circle [radius=1pt];
	\draw[fill] (9,0) circle [radius=1pt];
	\draw[fill] (9.5,0) circle [radius=1pt];
	\foreach \x in {8.5,9,9.5}{%
		\draw (9,2) -- (\x,0);
		}
	\draw[fill] (10,0) circle [radius=1pt];
	\draw[fill] (10.5,0) circle [radius=1pt];
	\draw[fill] (11,0) circle [radius=1pt];
	\foreach \x in {10,10.5,11}{%
		\draw (10.5,2) -- (\x,0);
		}
	\draw[fill] (11.5,0) circle [radius=1pt];
	\draw[fill] (12,0) circle [radius=1pt];
	\draw[fill] (12.5,0) circle [radius=1pt];
	\foreach \x in {11.5,12,12.5}{%
		\draw (12,2) -- (\x,0);
		}
	\draw[fill] (13,0) circle [radius=1pt];
	\draw[fill] (13.5,0) circle [radius=1pt];
	\draw[fill] (14,0) circle [radius=1pt];
	\foreach \x in {13,13.5,14}{%
		\draw (13.5,2) -- (\x,0);
		}
	%%%%%	EPSILONS		%%%%%
	\node (e1) [right] at (14.5,4) {\Large\textcolor{accent}{$\varepsilon_1$}};
	\node (e2) [right] at (14.5,2) {\Large\textcolor{accent}{$\varepsilon_2$}};
	\node (e3) [right] at (14.5,0) {\Large\textcolor{accent}{$\varepsilon_3$}};
	%%%%%	LABELS (Ps)		%%%%%
	%%%%% 	TOP			%%%%%
	\node [left] at (5,5) {$p_1$};
	\node [left] at (7.5,5) {$p_2$};
	\node [right] at (10,5) {$p_3$};

	%%%%%	SECOND SET	%%%%%
	\node at (2,3) {\small$p_1^+$};
	\node at (2.75,3) {\small$p_2^-$};
	\node [right] at (4,3) {\small$p_3^-$};

	\node at (6.5,3) {\small$p_1^-$};
	\node at (7.25,3) {\small$p_2^+$};
	\node [right] at (8.5,3) {\small$p_3^-$};

	\node at (11,3) {\small$p_1^-$};
	\node at (11.75,3) {\small$p_2^-$};
	\node [right] at (13,3) {\small$p_3^+$};

	%%%%%	THIRD SET	%%%%%
	\foreach \x in {1,2.5,4}{
		\node at (\x,.5) {\tiny$p_1^+$};
		\node at (\x+.5-.125,.5) {\tiny$p_2^-$};
		\node at (\x+1.125,.5) {\tiny$p_3^-$};}
\foreach \x in {5.5,7,8.5}{
	\node at (\x,.5) {\tiny$p_1^-$};
	\node at (\x+.5-.125,.5) {\tiny$p_2^+$};
	\node at (\x+1.125,.5) {\tiny$p_3^-$};}
\foreach \x in {10,11.5,13}{
	\node at (\x,.5) {\tiny$p_1^-$};
	\node at (\x+.5-.125,.5) {\tiny$p_2^-$};
	\node at (\x+1.125,.5) {\tiny$p_3^+$};}
	%%%%%	NUMERALS		%%%%%
	\foreach \x in {1,2.5,4,5.5,7,8.5,10,11.5,13}{%
		\node [below] at (\x,0) {$1$};
		}
	\foreach \x in {1.5,3,4.5,6,7.5,9,10.5,12,13.5}{%
		\node [below] at (\x,0) {$2$};
		}
	\foreach \x in {2,3.5,5,6.5,8,9.5,11,12.5,14}{%
		\node [below] at (\x,0) {$3$};
		}
	\node [left] at (3,4) {$1$};
	\node [left] at (7.5,4) {$2$};
	\node [right] at (12,4) {$3$};
	\foreach \x in {1.5,6,10.5}{%
		\node [left] at (\x,2) {$1$};
		}
	\foreach \x in {3,7.5,12}{%
		\node [left] at (\x,2) {$2$};
		}
	\foreach \x in {4.5,9,13.5}{%
		\node [right] at (\x,2) {$3$};
		}
\end{tikzpicture}